\documentclass[11pt]{article}
\usepackage[english]{babel}
\usepackage{amsmath, amsthm, amssymb, calrsfs, wasysym, verbatim, bbm, color, graphics, geometry}
\usepackage[pagewise]{lineno}
\newcommand\numberthis{\addtocounter{equation}{1}\tag{\theequation}}

\usepackage{titlesec}
\titlelabel{\thetitle.\quad}

\newcommand{\R}{\mathbb{R}}

\newcommand{\rmd}{\mathrm{d}}

\newtheorem{theorem}{Theorem}
\newtheorem{proposition}{Proposition}
\newtheorem{defn}{Definition}

\newtheorem{rem}{Remark}
\newtheorem{lemma}{Lemma}
\newtheorem{corollary}{Corollary}
\newtheorem{claim}{Claim}
\newtheorem{example}{Example}

\usepackage{hyperref}
\usepackage{cleveref}
\usepackage{xcolor}
\usepackage{verbatim}
\usepackage{hyperref}
\hypersetup{linkcolor={blue},citecolor={blue},urlcolor={red}}

\title{Strong maximum principle for fully nonlinear nonlocal problems}
\author{Juan Pablo Cabeza, Gabrielle Nornberg, Disson dos Prazeres }
\date{\today} 

\author{
Juan Pablo Cabeza\\
\small Departamento de Ingeniería Matemática, Universidad de Chile\\
\small Beauchef 851, Santiago, Chile\\
\small\texttt{jcabeza@dim.uchile.cl}
\and
Gabrielle Nornberg\\
\small Departamento de Ingeniería Matemática, Universidad de Chile\\
\small Beauchef 851, Santiago, Chile\\
\small \texttt{gnornberg@dim.uchile.cl
}
\and
Disson dos Prazeres\\
\small Departamento de Matemática, Universidade Federal de Sergipe\\
\small Av. Marechal Rondon s/n, Jd. Rosa Elze, São Cristóvão-SE, Brazil.\\
\small\texttt{disson@mat.ufs.br}
}

\begin{document}
\maketitle

\begin{abstract}

In this paper, we study solvability and qualitative properties of nonnegative solutions for a sublinear nonlocal problem with fully nonlinear structure in the form 
$$
  \mathcal{M}^{\pm}[u]+a(x)u^{q}(x)=0 \; \text{ in }\Omega,\qquad u\geq 0 \; \text{ in }\Omega. 
$$ 
Here $\Omega \subset \mathbb{R}^n$ is a bounded $C^{1,1}$ convex domain, $\mathcal{M}^{ \pm}$ stands for nonlocal Pucci extremal operators defined in a class $\mathcal{L}_*$ of homogeneous kernels, $q\in(0,1)$, and $a$ is a possibly sign-changing weight.

We introduce a new nonlocal hypothesis on the negative part of the solution outside the domain, which together with the negative part of the potential, influences the formation of dead cores and cannot be removed. 
Our approach relies on uniform bounds from below of the maximum of nontrivial solutions through Liouville theorems, and on a Hopf lemma for viscosity solutions driven by fully nonlinear operators, which we also prove.

\end{abstract}


\section{Introduction}

The strong maximum principle (SMP) is a fundamental tool in the study of elliptic PDEs. It states that a nonnegative supersolution in a domain cannot vanish within that domain, unless it vanishes identically, see~\cite{gilbarg1977elliptic}.
The nonvality of SMP is related to the the study of dead core solutions, that is, nontrivial solutions that vanish on a nonempty subset of the domain. The latter has been a topic of increasing interest of elliptic PDEs with application to the theory of chemical-biological processes, combustion dynamics, and related fields.
We quote the existence and nonexistence of nonnegative dead core solutions for fully nonlinear models \cite{da2022non}, fully coupled nonlinear systems \cite{araujo2024fully}, and improved regularity results for solutions to a dead core problem in \cite{ SS2018,dos2025improved, T2016}.
Dead cores for equations involving nonlocal operators present additional difficulties arising from the inherent nonlocality of the fractional operator, and will be the central role of this work.


In this paper we analyze existence, nonexistence, and qualitative properties of solutions that satisfy either a strong maximum principle or a dead-core configuration, for integro-differential equations in the form
\begin{equation}\label{eq:puccinonlocal}
    \mathcal{M}^{\pm}[u]+a(x)u^{q}(x)=0 \; \text{ in }\Omega,\qquad u\geq 0 \; \text{ in }\Omega.
\end{equation}
Here, \(\Omega\) is a bounded $C^{1,1} $ convex domain of \(\mathbb{R}^{n}\), \(a\in C(\overline{\Omega})\) is a possibly sign-changing weight, 
\(q\in(0,1)\), and \(\mathcal{M}^{\pm}\) are Pucci extremal operators given by
\begin{equation}\label{definitionpucci}
   \mathcal{M}^{+}[u](x) :=\sup_{L\in\mathcal{L}_{*}}Lu(x) \qquad\text{ and }\qquad \mathcal{M}^{-}[u](x) :=\inf_{L\in\mathcal{L}_{*}}Lu(x)
\end{equation}
defined in the class $\mathcal{L}_{*}$ consisting of the following linear operators 
\[Lu(x)=
\textstyle 
C_{n,s}\int_{\mathbb{R}^n}( u(x+y)+u(x-y)-2u(x) )\mu\left(\frac{y}{|y|}\right)\frac{\mathrm{d}y}{|y|^{n+2s}},\] with $\mu : \mathbb{S}^{n-1} \to [\lambda,\Lambda]$ satisfying $\mu(\theta)=\mu(-\theta)$.

We aim to investigate whether nontrivial solutions of our nonlocal problem \eqref{eq:puccinonlocal}, for which we establish sufficient conditions for existence (see Theorem~\ref{TheoremExistence} in Section 5), satisfy the SMP within this nonlocal framework.
More precisely, our goal is to study how the presence of the negative part of the solution outside the domain, and the negative part of the potential $a$, may influence the appearance of dead cores. 
We next state our main theorems on the nonexistence of dead cores. To this matter, we define the set \[\mathcal{P}^{0}(\Omega):=\{u\in C_{0}^{s}(\overline{\Omega}):u>0 \text{ in }\Omega, \;\partial^{s}_{\nu}u<0 \text{ on }\partial\Omega\}\]
and the norm
\begin{equation}\label{norm}
        \|u\|_{L^{1}_{s}(\mathbb{R}^{n})}=\int_{\mathbb{R}^{n}}\frac{|u(x)|}{1+|x|^{n+2s}}\mathrm{d}x .
    \end{equation}
   
\begin{theorem}\label{teo1} Let $\Omega$ be a bounded $C^{1,1}$ convex domain and $q\in (0,1)$.
Let \(a\in C(\overline{\Omega})\) such that $a^+\not\equiv 0$.  Then there exists \(\delta > 0\) so that if 
    $$
    \|u^{-}\|_{L^{1}_{s}(\mathbb{R}^{n})} < \delta,
    $$
    and
$$
    |u(x)|\leq C\sup_\Omega u \,(1+|x|)^{1+\alpha} \mbox{ in }\mathbb{R}^n \textrm{ for some $\alpha$}, 
$$  
    then every nontrivial solution of \eqref{eq:puccinonlocal} such that $u \leq 0 $ in $ \mathbb{R}^{n}\backslash\Omega$ 
    belongs to \(\mathcal{P}^{0}(\Omega)\).
\end{theorem}

Note that the only way to have dead core solutions is by assuming that $u$ has a nontrivial negative part. In fact, note that if \(u\geq 0\) in \(\mathbb{R}^{n}\) is a nontrivial solution, even in the case that \(a\) changes sign, then \(u>0\) in \(\Omega\), as in the proof of the standard SMP for nonlocal operators (if  $u$ vanishes in $\Omega$ is a $C^{1,1}$ function, then there exists an interior minimum point \(x_{0}\in\Omega\) such that \(u(x_0)=\inf_{\Omega}u=0\), so for $\mathcal{M}^-$ we have
\begin{center}$
   0 = \mathcal{M}^{-}[u](x_0) = \inf_{L\in\mathcal{L}}Lu(x_{0})  
  \ge 2 \lambda  C_{n,s}\text{P.V. }\int_{\mathbb{R}^{n}}\frac{u(y)}{|y-x_0|^{n+2s}}\mathrm{d}y >0,
$\end{center}
which is impossible). 

Furthermore, when the weight \(a\) has a sign, it is not necessary to impose a sign condition on \(\mathbb{R}^{n}\backslash\Omega\), even though a smallness condition in the weighted \(L_{s}^{1}-\)norm of $u^-$ is still required.

\begin{theorem}\label{teo2}
Let $\Omega$ be a bounded $C^{1,1}$ convex domain and $q\in (0,1)$.
Let \(a\in C(\overline{\Omega})\) such that $a> 0$ in $\Omega$.  Then there exists  \(\delta > 0\) so that if 
if 
    $$
    \|u^{-}\|_{L^{1}_{s}(\mathbb{R}^{n})} < \delta,
    $$
    and
$$
    |u(x)|\leq C\sup_\Omega u \,(1+|x|)^{1+\alpha} \mbox{ in }\mathbb{R}^n \textrm{ for some $\alpha$}, 
$$ 
    then every nontrivial solution of \eqref{eq:puccinonlocal} belongs to \(\mathcal{P}^{0}(\Omega)\).
\end{theorem}

Observe that in the local setting, the above result is trivial when \(a\) is positive; however, even in this particular case, the nonlocal SMP may still be violated, as shows our Example \ref{counterexample} in the final section.

The proofs and hypotheses of our theorems rely fundamentally on certain properties which in nature differ from the local structure. 
In particular, one of our key tools ensures that the maxima of nontrivial solutions stay uniformly bounded below through Liouville theorems. 
We also establish a Hopf lemma for Pucci extremal operators that extends the one in \cite{dipierro2024fractional} (see also \cite{DelTP2025, OS2025}) under slightly different assumptions, which is of independent interest. 
As a dichotomy, we also derive regularity results for the solutions of extremal operators that do not depend on any prescribed growth condition. 

The rest of the paper is organized as follows: Section 2 presents auxiliary results and notation. In Section~3 we prove a Hopf lemma and a strong maximum principle for our class of fully nonlinear equations. Section 4 explores the applications of this lemma to our sublinear equations  involving fully nonlinear nonlocal operators. Finally, in Section 5 we give sufficient conditions for existence of nontrivial solutions to our problem, by also presenting the necessity of our hypothesis towards the violation of the strong maximum principle.


\section{Auxiliary tools}\label{auxiliary}

We start the section by recalling the definition of viscosity solution, as in~\cite{fernandez2024integro}.
Here, \(u\in L_{s}^{1}(\mathbb{R}^{n})\) if $  \|u\|_{L^{1}_{s}(\R^n)}<\infty$, where $\|u\|_{L^{1}_{s}(\R^n)}$ is given by \eqref{norm}.
\begin{defn}\label{def-visc}
    Let \(s\in(0,1)\), \(\Omega\subset\mathbb{R}^{n}\) be an open set, and let \(f\in C(\Omega)\). Consider the equation
    \begin{equation}\label{viscosityPucci}
        \mathcal{M}^{\pm}[u](x) = f(x)
    \end{equation}
    \begin{itemize}
        \item Let \(u\in L_{s}^{1}(\mathbb{R}^{n})\) be an upper semi continuous function, we say \(u\) is a viscosity subsolution to~\eqref{viscosityPucci} in \(\Omega\) and we denote \(\mathcal{M}^{\pm}[u](x)\geq f(x)\) in \(\Omega\) , if for any \(x\in\Omega\) and any neighborhood \(N_{x}\subset\Omega\) of \(x\) in \(\Omega\), and for any test function \(\phi\in L^{1}_{s}(\mathbb{R}^{n})\) such that \(\phi\in C^{2}(N_x)\), \(\phi(x) = u(x)\), and \(\phi\geq u\) in all of \(\mathbb{R}^{n}\), we have \(\mathcal{M}^{\pm}[\phi](x)\geq f(x)\).
        \item Let \(u\in L_{s}^{1}(\mathbb{R}^{n})\) be a lower semi continuous function, we say \(u\) is a viscosity supersolution to~\eqref{viscosityPucci} in \(\Omega\) and we denote \(\mathcal{M}^{\pm}[u](x)\leq f(x)\), if for any \(x\in\Omega\) and any neighborhood \(N_{x}\subset\Omega\) of \(x\) in \(\Omega\), and for any test function \(\phi\in L^{1}_{s}(\mathbb{R}^{n})\) such that \(\phi\in C^{2}(N_x)\), \(\phi(x) = u(x)\), and \(\phi\leq u\) in all of \(\mathbb{R}^{n}\), we have \(\mathcal{M}^{\pm}[\phi](x)\leq f(x)\). 
        \item We say that \(u\in C(\Omega)\cap L_{s}^{1}(\mathbb{R}^{n})\) is a viscosity solution to~\eqref{viscosityPucci} in \(\Omega\), and we denote \(\mathcal{M}^{\pm}[u](x)=f(x)\) in \(\Omega\) if it is both a viscosity subsolution and supersolution.
    \end{itemize}
\end{defn}

Next we recall a barrier result established in \cite{ros2016boundary} for the class $\mathcal{L}_{*}$, which will be employed onward in the text. 
\begin{lemma}\label{rho}
    Let \(s_0\in(0,1)\), \(s\in[s_0,1)\), and define
\begin{align*}
    \rho_{1}(x) = (\mathrm{dist}(x, \mathbb{R}^{n}\backslash B_{1}))^{s},  
    \quad
    \rho_{2}(x) = (\mathrm{dist}(x, \mathbb{R}^{n}\backslash B_{1}))^{3s/2}. 
\end{align*} Then, for all \(x\in B_{1}\backslash B_{1/2}\) we have
    \begin{align*}
        0  \geq \mathcal{M}^{-}[\rho_{1}](x) &\geq -C\{1+(1-s)|\log(1-|x|)|\}, \\
        \mathcal{M}^{-}[\rho_{2}](x) &\geq c(1-|x|)^{-s/2} - C,
    \end{align*}
    where the constants \(c>0\) and \(C\) depend only on \(n\), \(s_0\) and the ellipticity constants.
\end{lemma}

We now derive a construction of a suitable subsolution for nonlocal operators perturbed by a potential term \(V\), that we will use in the following section. 

\begin{lemma}\label{supersolucion}
Let $s_0\in(0,1)$ and $s\in[s_0,1)$. There exists a radial  bounded continuous function $\varphi$ so that
\begin{equation}\label{rosoton1}
\begin{cases}
    \mathcal{M}^{-}[\varphi](x) - V^{-}(x)\varphi(x) \geq c & \text{ in } B_{1} \backslash B_{1/2}, \\ 
    \varphi(x) = 0 & \text{ in } \mathbb{R}^{n} \backslash B_{1}, \\
    \varphi(x) \geq c(1-|x|)^{s} & \text{ in } B_{1}, \\
    \varphi(x) \leq 1 & \text{ in } \overline{B_{1/2}},
\end{cases}
\end{equation}
for a positive constant $c$ that depends only on $n$, $s_0$ and the ellipticity constants.
\end{lemma}

\begin{proof}
    Consider the function \(\rho(x) :=\rho_{1}(x)+\rho_{2}(x)\); we proceed to construct a subsolution in the annulus \(B_{1}\backslash \overline{B_{1-\epsilon}}\). By Lemma~\ref{rho}, for \(|x|\in \left(\frac{1}{2},1\right)\) we have \[\mathcal{M}^{-}[\rho](x)-V^{-}(x)\rho(x)\geq -C\{1+(1-s)|\log(1-|x|)|\} + c(1-|x|)^{-s/2} - V^{-}(x)\left(\rho_{1}(x) + \rho_{2}(x)\right)-C,\]
    since \(\rho_{1}(x)=(1-|x|)^{s}\) and \(\rho_{2}(x) = (1-|x|)^{3s/2}\). Observing that both \(\rho_1\) and \(\rho_2\) vanish outside \(B_1\), and \(V^{-}\in L^{\infty}(B_{1})\), we conclude that \(-V^{-}(x)\rho(x)\geq - \|V^{-}\|_{\infty}\displaystyle\sup _{x\in B_{1}}\rho(x) = -2\|V^{-}\|_{\infty}\). It follows that
    \begin{equation*}
        \mathcal{M}^{-}[\rho](x)-V^{-}(x)\rho(x) \geq -C\{1+(1-s)|\log(1-|x|)|\} + c(1-|x|)^{-s/2} - 2\|V^{-}\|_{\infty} - C.
    \end{equation*}
    Hence, we can take \(\epsilon>0\) small enough such that \(\mathcal{M}^{-}[\rho] - V^{-}(x)\rho\geq1\) in \(B_{1}\backslash\overline{B_{1-\epsilon}}\), namely 
    \begin{equation*}
        \epsilon > \left( C_{1} + C_{2}|\log(1-|x|)|\right)^{-2/s},
    \end{equation*}
    with \(C_{1}:=c^{-1}\left(1+2C+2\|V^{-}\|_{\infty}\right)\) and \(C_{2}:=c^{-1}(1-s)C\).

    Given a large integer \(N\) and \(C>1\), we define \(0\leq k\leq N\), \(\displaystyle\widetilde{\rho}(x):=\max_{0\leq k\leq N} C^{k}\rho(2^{k/N}x)\) and \[A_{k} = \{x\in B_{1}:\widetilde{\rho}(x) = C^{k}\rho(2^{k/N}x)\}.\] 
    Since \(A_0 \subset B_{1}\backslash \overline{B_{1-\epsilon}}\), it follows that \(\mathcal{M}^{-}[\widetilde{\rho}]-V^{-}(x)\widetilde{\rho} \geq 1\) in \(A_0\). Further, we have \(A_k = 2^{-k/N}A_0\). For any \(x\in A_k\), it holds that \(2^{k/N}x\in A_0\subset B_{1}\backslash \overline{B_{1 - \epsilon}}\), thus
    \begin{align*}
        \mathcal{M}^{-}[\rho(2^{k/N})](x) &= 2^{2ks/N}\mathcal{M}^{-}[\rho](2^{k/N}x) 
        \geq \mathcal{M}^{-}[\rho](2^{k/N}x).
    \end{align*}
    For any \(x\in A_k\) we obtain the following estimate
    \begin{equation*}
        \mathcal{M}^{-}[\widetilde{\rho}](x) - V^{-}(x)\widetilde{\rho}(x) \geq C^{k}(\mathcal{M}^{-}[\widetilde{\rho}](2^{k/N}x) - V^{-}(2^{k/N}x)) \geq 1.
    \end{equation*}
    Finally, we set \(\varphi = c\widetilde{\rho}\) with \(c>0\) sufficiently small so that \(\varphi\leq 1\) in \(\overline{B_{1/2}}\).
\end{proof}


\section{Hopf Lemma and Strong Maximum Principle}

In this section, we present a version of Hopf lemma adapted to our nonlocal fully nonlinear context. This result will play an important role in the analysis of equation~\eqref{eq:puccinonlocal}, particularly in establishing the boundary behavior of nonnegative solutions. 

As a preliminary step, in light of \cite{dipierro2024fractional}, the following proposition will be instrumental.
\begin{proposition}\label{lema2}
    Let \(\Omega\) be an open subset of \(\mathbb{R}^{n}\), \(x_0\in\partial\Omega\) and \(u:\mathbb{R}^{n}\to\mathbb{R}\) be a continuous function such that \(u\geq 0\) in a ball \(B_{R}(x_0)\). Consider a subset \(S\Subset B_{R}(x_0)\), and denote \(d_0 = \mathrm{dist}(S, \partial B_{R}(x_0))\), then there exists a constant \(\widetilde{C}:=\widetilde{C}(n,s,\lambda,\Lambda)>0\) such that \[\mathcal{M}^{+}[u^{-}] \leq \widetilde{C}\|u^{-}\|_{L^{1}_{s}},\quad\forall x\in S.\]
\end{proposition}

\proof Let \(x\in S\), then \(x\in\{u \geq 0\}\) it follows that \(u^{-}(x) = \max\{-u(x),0\} = 0\). Furthermore, it can be observed that if \(y\in\mathbb{R}^{n}\) with \(u^{-}(x+y)>0\), then \(x+y\in\mathbb{R}^{n}\backslash B_{R}(x_0)\). In particular, we obtain \(|y| = |(x+y)-x|\geq d_0\). Now,
\begin{align*}
  Lu^{-}(x) &= (1-s)\int_{\mathbb{R}^{n}}
    ( {u^{-}(x+y) + u^{-}(x-y)} - 2u^{-}(x) )\mu \left(\frac{y}{|y|}\right)\frac{\mathrm{d}y}{|y|^{n+2s}} \\
 &\leq (1-s) 
    \Lambda\int_{\mathbb{R}^{n}\backslash B_{d_0}} 
    ({u^{-}(x+y) + u^{-}(x-y)} )\frac{\mathrm{d}y}{|y|^{n+2s}}.
\end{align*}
In order to get an appropriate bound for this integral, we observe that $|x\pm y|^{n+2s} \leq \left(|2\max\{|x|,|y|\}|\right)^{n+2s}=2^{n+2s}|y|^{n+2s},$ from which we can deduce 
\begin{center}
$1+|x+y|^{n+2s} \leq \left(\frac{|y|}{d_{0}}\right)^{n+2s} + 2^{n+2s}|y|^{n+2s} = |y|^{n+2s}\left(\frac{1}{d_{0}^{n+2s}} + 2^{n+2s}\right).$
\end{center}
Consequently, 
   $ \frac{1}{|y|^{n+2s}} \leq \frac{M_{n}}{1+|x+y|^{n+2s}},$ with $ M_{n}:=\left(\frac{1}{d_{0}^{n+2s}} + 2^{n+2s}\right).$
Considering these estimates, we express the integral in terms of the norm with weight of the negative part of \(u\),
\begin{align*}
 Lu^{-}(x) \leq {(1-s)}M_{n}\Lambda \int_{\mathbb{R}^{n}\backslash B_{d_0}}\left(\dfrac{u^{-}(x+y)}{1+|x+y|^{n+2s}} + \dfrac{u^{-}(x-y)}{1+|x-y|^{n+2s}}\right)\mathrm{d}y \leq {(1-s)}M_{n}\Lambda\|u^{-}\|_{L^{1}_{s}}.
\end{align*}
Therefore, 
$
    \mathcal{M}^{+}[u^{-}] \leq \widetilde{C}\|u^{-}\|_{L^{1}_{s}},
$
where \(\widetilde{C}:= {(1-s)}M_{n}\Lambda>0\).
\endproof

We are now in a position to state and prove the following version of the Hopf lemma. 
Given an open set $\Omega$ of $\mathbb{R}^{n}$, $n \geq 1$, we say $\Omega$ satisfies the interior ball condition at $x_{0}\in\partial\Omega$ if there is $r_{0}>0$ such that, for every $r\in(0,r_{0}]$ there exists a ball $B_{r}(x_{r})\subset \Omega$ with $x_{0} \in \partial B_{r}(x_{r})\cap\partial\Omega$.

\begin{theorem}\label{thm1}
    Let $\Omega \subset \mathbb{R}^{n}$ be an open subset and $x_{0}\in\partial\Omega$. Assume that $\partial\Omega$ satisfies the interior sphere condition at $x_0$. 
Let $u\in L_{s}^{1}(\R^n)$ be a lower semicontinuous function such that $u(x_{0})=0$ and $u^{-}\in L^{\infty}(\mathbb{R}^{n})$ and 
    \begin{equation*}
        \mathcal{M}^{-}[u](x) + V(x)u \leq 0 \qquad\text{ in }\Omega
    \end{equation*}
    in the viscosity sense, where $V\in L^{1}_{loc}(\Omega)$ is such that $V^{-}\in L^{\infty}(\Omega)$.
    Suppose that there exists $R>0$ such that $u\geq 0$ in $B_{R}(x_{0})$, $u>0$ in $\Omega\cap B_{R}(x_{0})$, and for some \(r\) small enough holds the following inequality
    \begin{equation}\label{inequalityhopf}
        \widetilde{C}\|u^{-}\|_{L^{1}_{s}} < \dfrac{\alpha_r}{r^{2s}}c,
    \end{equation}
    with \(C,c>0\) constants and \(\alpha_r = \inf_{B_{r/2}(x_r)}u\). Then, for every $\beta\in \left(0,\frac{\pi}{2}\right)$ we have 
    \begin{equation}\label{hopfslemma}
        \liminf_{x\to x_0, x\in\Omega}\frac{u(x)}{|x-x_{0}|^{s}}>0,
    \end{equation}
    whenever the angle between $x-x_0$ and the vector joins $x_0$ and the center of the interior sphere is smaller than $\pi/2-\beta$.
\end{theorem}

\begin{proof}
The idea of the proof is to apply a comparison argument between the function \(u\) and a barrier function \(\psi_r\) in an annulus \(A_r\), with \(u(x) \geq 0\) for all \(x\in A_r\). Since the barrier a smooth function, one may apply the argument as it $u$ was a regular solution, up to considering \(\psi_r\) as a test function. 

Let us consider the negative part of the potential, \(V^{-}(x)\), in place of \(V(x)\), because if \(V = V^{+}-V^{-}\) in the annulus \(A_r\) with \(V^{+}(x)u \geq 0\), it follows that
\begin{align*}
    \mathcal{M}^{-}[u]-V^{-}(x)u &\leq \mathcal{M}^{-}[u] + V^{+}(x)u - V^{-}(x)u \\
   & = \mathcal{M}^{-}[u] + V(x)u.
\end{align*}

Let \(x_0\) be a boundary point of the domain \(\Omega\), and suppose the interior ball condition holds at \(x_0\). This means that for every sufficiently small radius \(r>0\), there exists a point \(x_r\) inside \(\Omega\) such that the open ball \(B_{r}(x_r)\) lies entirely inside \(\Omega\) and is tangent to the boundary at \(x_0\). Now, for any radius \(r>0\), we consider any \(x\in B_{r}(x_r)\) and set \(\alpha_r:=\inf_{B_{r/2}(x_r)}u\), 
\begin{equation*}
    \psi_{r}(x)=\alpha_{r}\varphi\left(\frac{x-x_{r}}{r}\right),
\end{equation*}
where \(\varphi\) and \(C>0\) be as in Lemma~\ref{supersolucion}. Given that \(B_{r}(x_r)\Subset\Omega \cap B_{R}(x_0)\), then \(u > 0\) and we conclude that \(\alpha_{r}>0\). Furthermore, for \(x\in B_{r}(x_r)\backslash \overline{B_{r/2}(x_r)}\), and an arbitrary kernel \(a\),
\begin{align*}
    L\psi_r(x) &= (1-s)\int_{\mathbb{R}^{n}}
    ({\psi_{r}(x+y) + \psi_{r}(x-y)}-2\psi_{r}(x))\mu \left(\dfrac{y}{|y|}\right)\dfrac{\mathrm{d}y}{|y|^{n+2s}} \\
    &=\alpha_{r}(1-s)\int_{\mathbb{R}^{n}}
    \left[ {\varphi(\frac{x+y-x_{r}}{r}) + \varphi(\frac{x-y-x_{r}}{r})}-2\varphi\left(\frac{x-x_r}{r}\right) \right]
    \mu\left(\dfrac{y}{|y|}\right)\dfrac{\mathrm{d}y}{|y|^{n+2s}} \\ 
    &= \frac{\alpha_{r}(1-s)}{r^{2s}}\int_{\mathbb{R}^{n}}
    \left( \varphi(x'+y') + \varphi(x'-y')-2\varphi\left(x'\right)\right)
    \mu \left(\dfrac{y'}{|y'|}\right)\dfrac{\mathrm{d}y'}{|y'|^{n+2s}} 
    = \frac{\alpha_{r}}{r^{2s}} L\varphi(x'),
\end{align*}
where we apply the change of variables \(x'=\frac{x-x_r}{r}\) and \(y'=\frac{y}{r}\). For the potential, we have
\begin{align*}
    V^{-}(x)\psi_{r}(x) = \alpha_{r}V^{-}(x_r+rx')\varphi(x')=\alpha_{r}r^{-2s}\widetilde{V^{-}}(x')\varphi(x'),
\end{align*}
where \(\widetilde{V^{-}}(x')=r^{2s}V^{-}(x_{r}+rx')\). It follows that \[|\widetilde{V^{-}}(x')|=r^{2s}\left|V^{-}(x_{r}+rx')\right|\leq r^{2s}\|V^{-}\|_{\infty}\leq\|V^{-}\|_{\infty}.\] Now, by applying Lemma~\ref{supersolucion}, we obtain that for any \(x\in B_{r}(x_r)\backslash B_{r/2}(x_r)\)
\begin{align*}
    \mathcal{M}^{-}[\psi_r](x) - V^{-}(x)\psi_{r}(x) = \frac{\alpha_{r}}{r^{2s}}\inf_{L\in\mathcal{L}_{*}}L\varphi(x') - \alpha_{r}r^{-2s}\widetilde{V^{-}}(x')\varphi(x') 
    \geq \frac{\alpha_{r}}{r^{2s}}c. 
\end{align*}
In addition, we observe that
\begin{align*}
    \psi_{r}(x)=\alpha_{r}\varphi\left(\frac{x-x_{r}}{r}\right) \geq \alpha_{r}c\left(1-\left|\frac{x-x_{r}}{r}\right|\right)^{s} = c\dfrac{\alpha_r}{r^{s}}\left( r-|x-x_r|\right)^{s},\quad x\in B_{r}(x_r).
\end{align*}
Also, we have \(\psi_r \leq \alpha_r\) in \(\overline{B_{r/2}(x_r)}\). In this way, we have established the following scaled barrier function:
\begin{equation}
\begin{cases}
    \mathcal{M}^{-}[\psi_r] - V^{-}(x)\psi_{r}(x) \geq \frac{\alpha_r}{r^{2s}}c & \text{ in } B_{r}(x_r) \backslash B_{r/2}(x_r), \nonumber\\ 
    \psi_r = 0 & \text{ in } \mathbb{R}^{n} \backslash B_{r}(x_r), \nonumber\\
    \psi_r \geq c\frac{\alpha_r}{r^s}(r-|x-x_r|)^{s} & \text{ in } B_{r}(x_r), \numberthis \label{scaling}\\
    \psi_r \leq \alpha_r & \text{ in } \overline{B_{r/2}(x_r)}. 
\end{cases}
\end{equation}

Now, we aim to show that the function \(w := \psi_r - u^{-}\) satisfies the inequality \(u \geq w\) via comparison principle. First, note that, in \(\mathbb{R}^{n}\backslash B_{r}(x_{r})\), the function \(w\) is bounded above by \(u\), since $w = - u^{-} = u - u^{+} \leq u.$ Additionally, in \(\overline{B_{r/2}(x_r)}\) we have \[w \leq \alpha_r - u^{-} \leq \inf_{B_{r/2}(x_r)}u \leq u.\]

Let us consider \(S = B_{r}(x_r) \backslash\overline{B_{r/2}(x_r)}\), observe that the distance between \(B_{r}(x_r)\) and \(\partial B_{R}(x_0)\) is greater than \(R/2\), whenever \(r>0\) is taken small enough. Thus, following Proposition~\ref{lema2}, the constant \(\widetilde{C}>0\) does not depend on \(r\). Consequently, applying Proposition~\ref{lema2} with \(S\), together with the standard inequalities for Pucci extremal operators, for all \(x\in B_{r}(x_r) \backslash\overline{B_{r/2}(x_r)}\), and taking into account that \(u^{-}(x)=0\) for \(x\in S\), we have
\begin{align*}
    \mathcal{M}^{-}[w](x) - V^{-}(x)w(x) &\geq \mathcal{M}^{-}[\psi_r](x) + \mathcal{M}^{-}[-u^{-}](x) - V^{-}(x)\psi_{r}(x) \\
    &= \mathcal{M}^{-}[\psi_r](x) - \mathcal{M}^{+}[u^{-}](x) - V^{-}(x)\psi_{r}(x)\\
    &\geq \dfrac{\alpha_r}{r^{2s}}c - \widetilde{C}\|u^{-}\|_{L^{1}_{s}} \geq 0,
\end{align*}
where in the last inequality we used~\eqref{inequalityhopf}.
Therefore for all \(x\in B_{r}(x_r) \backslash\overline{B_{r/2}(x_r)}\),
\begin{align*}
    \mathcal{M}^{-}[u-w](x) - V^{-}(x)(u-w)(x) &\leq \mathcal{M}^{-}[u](x) + \mathcal{M}^{+}[-w](x) - V^{-}(x)u(x) + V^{-}(x)w(x) \\
    &\leq -\mathcal{M}^{-}[w](x) + V^{-}(x)w(x) \leq 0.
\end{align*}

In summary, we have that, 
\begin{equation*}
    \begin{cases}
    \mathcal{M}^{-}[u-w] - V^{-}(x)(u - w) \leq 0 & \text{ in } B_{r}(x_r)\backslash\overline{B_{r/2}(x_r)}, \\ 
    u \geq w & \text{ in }\mathbb{R}^{n}\backslash ( B_{r}(x_r)\backslash\overline{B_{r/2}(x_r)} ),
    \end{cases}
    \end{equation*}
and the comparison principle yields that \(u\geq w\) in \(B_{r}(x_r)\backslash\overline{B_{r/2}(x_r)}\).

To finish the proof, we follow the same reasoning as in \cite{dipierro2024fractional}, for $\beta\in(0,\pi/2)$, let 
\begin{equation}
    \mathcal{C}_{\beta} := \left\{x\in\Omega:\dfrac{x-x_0}{|x-x_0|}\cdot\nu >c_{\beta}\right\},
\end{equation}
where $c_{\beta}:=\cos\left(\frac{\pi}{2}-\beta\right)>0$ and $\nu$ is the normal vector joining $x_0$ with the center of the interior ball. 

Take any sequence $x_k\in \mathcal{C}_\beta$ such that $x_k\to x_0$ and for $k$ is sufficiently large, one has $|x_{k}-x_{0}|<\min\{2rc_{\beta}, r/2\}$. On the one hand, writing $x_r=x_0 +r\nu$, as in \cite{dipierro2024fractional},
\begin{equation*}
    |x_k-x_r|^2=|x_k-x_0-r\nu|^2\leq r^2-|x_k-x_0|(2c_\beta r-|x_k-x_0|)<r^2.
\end{equation*}
On the other hand,
$
    |x_k-x_r|\geq |x_r-x_0|-|x_k-x_0|=r-|x_k-x_0|>\frac{r}{2},
$
for $k$ large enough, we have $x_k\in B_{r}(x_r)\backslash \overline{B_{r/2}(x_r)}$. Next, since $r+|(x_k-x_0)-r\nu|\leq 2r + \frac{r}{2}=\frac{5}{2}r$,
\begin{align*}
    u(x_k) &\geq \textstyle w(x_k) 
    = \psi_r(x_{k}) \geq \frac{\alpha_r}{r^s}(r-|x_k-x_r|)^{s} = \frac{\alpha_r}{r^s}(r-|(x_k-x_0)-r\nu|)^{s} \\
    & \textstyle \geq \left(\frac{2}{5}\right)^{s}\frac{\alpha_r}{r^{2s}}(r^2-|(x_k-x_0)-r\nu|^2)^{s} \\
    & \textstyle \geq \left(\frac{2}{5}\right)^{s}\frac{\alpha_r}{r^{2s}}(2c_{\beta}r|x_{k}-x_{0}| -|x_{k}-x_{0}|^{2})^{s} > \left(\frac{2}{5}\right)^{s}\frac{\alpha_r}{r^{2s}}(2c_{\beta}r)^{s} |x_{k}-x_0|^{s}.
\end{align*}
Therefore,
\begin{equation*}
 \textstyle   \liminf_{k\to\infty}\frac{u(x_k)}{|x_k-x_0|^s}\geq\left(\frac{4}{5}\right)^{s}\frac{\alpha_r c_{\beta}^s}{r^s}.
\end{equation*}

This concludes the proof of Theorem \ref{thm1}.
\end{proof}

We now establish a version of the strong maximum principle as a consequence of fractional Hopf lemma. Although there exists a simpler argument that avoids the use of Hopf lemma, the proof we present here is more natural since it follows the step of the classical proof of Strong maximum principle for the local case.

\begin{proposition}[{Strong Maximum Principle}]\label{SMP} Let \(u\in L^{1}_{s}(\mathbb{R}^{n})\). Then, for
\begin{equation*}
    \begin{cases}
    \mathcal{M}^{-}[u] + V(x)u \leq 0  & \text{ in } \Omega, \\ 
    u \geq 0 & \text{ in }\mathbb{R}^{n},
    \end{cases}
\end{equation*}
implies that \(u>0\) in \(\Omega\) or \(u\equiv 0\) in \(\Omega\).
\end{proposition}

\proof Set \(\Omega_{0} = \{x\in B_{1}, u(x)=0\}\), \(\Omega^{+} = \{x\in B_{1}, u(x) > 0\}\). Suppose, for the sake of contradiction, that both sets \(\Omega_{0}\) and \(\Omega^{+}\) are nonempty.
Let \(\widetilde{x} \in \Omega^{+}\), and consider the largest ball \(B_{R}(\widetilde{x})\subset \Omega^{+}\), \(R>0\) on which \(\partial B_{R}(\widetilde{x})\cap \partial\Omega_{0} \neq \emptyset\), we can find a point \(x_0 \in \partial B_{R}(\widetilde{x})\cap \partial\Omega_{0} \) such that \(u=0\) in that region. Since \(x_0\in\Omega\) is a minimum point in the interior, 
\begin{equation*}
  \textstyle  \liminf_{x\to x_0, x\in\Omega}\frac{u(x) - u(x_0)}{|x-x_0|} = 0.
\end{equation*}

On the other hand, the fact that the solution is nonnegative everywhere implies that it maintains this sign in a whole neighbourhood of \(x_0\). In particular, within the portion of this neighborhood which belong to \(B_{R}(\widetilde{x})\), the function is positive, since we are located inside the region of positivity, \(\Omega^{+}\). Then, and we are precisely under the conditions in which the Hopf lemma applies. Thus, $\liminf\frac{u(x)}{|x-x_{0}|^{s}}>0.$ But then, 
\begin{equation*}
 \textstyle   \liminf_{x\to x_0, x\in\Omega}\frac{u(x) - u(x_0)}{|x-x_0|} = \liminf_{x\to x_0, x\in\Omega}\frac{u(x) - u(x_0)}{|x-x_0|^{s}}\frac{1}{|x-x_0|^{1-s}} = +\infty,
\end{equation*}
a contradiction. Therefore, one of the two sets must be empty.
\endproof

If in the previous proposition we remove the assumption that \( u \geq 0 \) in \( \mathbb{R}^{n} \), the problem becomes more delicate, and it does not descent for regions where \( u = 0 \) even though \( u \) is not identically zero in \( \Omega \). In this case, the following theorem guarantees higher regularity for supersolutions at points on \( \partial\{u > 0\} \) that can be touched by a ball from within the positivity set.

\begin{proposition} Let \(u\in L^{1}_{s}(\mathbb{R}^{n})\cap C^{0,\beta}(\overline{\Omega})\) for some \(\beta>0\), and \[T=\{x\in \partial\{u>0\}\cap\Omega:x\in \partial B_{r_0}(x_0) \mbox{ with }B_{r_0}(x_0) \subset \{u>0\}\}.\]
If \(u\geq 0\) in \(\Omega\) is a solution of $\mathcal{M}^{-}[u] + V(x)u \leq M$ in $ \Omega,$ for \(M>0\), then \[|u(y)|\leq C|y-x|^{2s} \mbox{ for } y\in \Omega \mbox{ and } x\in T,\] where \(C>0\) depends on the parameters \(n,\lambda,\Lambda,s\).
\end{proposition}

\begin{proof} Suppose that $u\neq0$ in $\Omega$ and that for some $x_1\in T$ there exist sequences $z_n$ and $r_n\rightarrow0$ such that $x_1\in \partial B_{r_n}(z_n)$, with $B_{r_n}(z_n)\subset\{u>0\}\cap \Omega $ and \[\frac{u(z_n)}{r_n^{2s}}>c>0.\] Moreover, without loss of generality, we may assume the balls $B_{r_n/2}(z_n)$ are disjoint. Since $u\in C^{0,\beta}(\overline{\Omega})$, then for \(y\in B_{r_n/2}(z_n)\) we get $u(y)\geq cr^{2s}.$ Therefore
\begin{equation*}
    \displaystyle\int_{\Omega}\frac{u(y)}{|y-x_1|^{n+2s}}\mathrm{d}y\geq \int_{\cup B_{r_n/2}(z_n) }\frac{u(y)}{|y-x_1|^{n+2s}}\mathrm{d}y\geq \int_{\cup B_{r_n/2}(z_n) }\frac{cr^{2s}}{|y-x_1|^{n+2s}}\mathrm{d}y.
\end{equation*}
Thus
\begin{equation*}
    \int_{\Omega}\frac{u(y)}{|y-x_1|^{n+2s}}\mathrm{d}y\geq \sum_{n} \int_{ B_{r_n/2}(z_n) }\frac{cr^{2s}}{|y-x_1|^{n+2s}}\mathrm{d}y=\infty.
\end{equation*}
On the other hand,
$$
\mathcal{M}^{-}[u](x_1)\leq M,
$$
which leads to a contradiction, and therefore proves the result.
\end{proof}
As a consequence, we obtain the following regularity result in the radial scenario.

\begin{corollary}
Let \( u \in L^{1}_{s}(\mathbb{R}^{n})\cap C^{0,\beta}(\Omega) \). If \( u \geq 0 \) in \( \Omega \) is a radial solution of
$
\mathcal{M}^{-}[u] + V(x) u \leq M$ in $ \Omega,$
then
\[
|u(y)| \leq C |y - x|^{2s}, \quad \text{for } y \in \Omega \text{ and } x \in \partial\{u > 0\} \cap \Omega.
\]
\end{corollary}


\section{Strong Maximum principle for sublinear equations}\label{secciondeadcore}

In this section we study the validity of the strong maximum principle (SMP)  for the sublinear problem
\begin{equation*}\label{eq:problem}\tag{$P_{a,q}$}
\begin{cases}
    \mathcal{M}^{-}[u](x) + a(x)u^{q}(x) = 0 & \text{ in } \Omega, \\ 
    u\geq 0 & \text{ in }\Omega, \\
     u \leq 0 & \text{ in } \mathbb{R}^{n}\backslash\Omega, \\
\end{cases}
\end{equation*}
where \(\Omega\subset\mathbb{R}^n\) is a bounded smooth domain, with $s, q\in(0,1)$ and $a\in C(\overline{\Omega})$ is a possible sign-changing function with $\Omega^+_a\neq \emptyset$. 
Here, $\Omega^+_a=\{x\in \Omega: a(x)>0\}$. 
To fix the ideas we will be considering  the extremal operator $\mathcal{M}^{-}$, but everything can be performed analogously for  
$\mathcal{M}^{+}$ instead.

We start the section with a localizing lemma that holds in both superlinear and sublinear scenarios. It establishes that every nontrivial solution of \eqref{eq:problem} assumes its positive maximum in some component of $\Omega^{+}_a$.

\begin{lemma}\label{lemma 5.1}
    Let \(u\) be a nontrivial solution of \eqref{eq:problem} with $q>0$. If $u\leq 0$ in $\mathbb{R}^n\setminus \Omega$, then a maximum point of $u$ where  \(u(x_0)=\max_{\bar\Omega} u\) can only be reached at a point \(x_0\in\Omega^{+}_a\).
\end{lemma}
 \begin{proof}
  Assume by contradiction that \(u(x_1) = \max_{\bar\Omega } u\)  with $a(x_1)\leq 0$. Since \(u\) is nonnegative in \(\Omega\), it follows that $\mathcal{M}^{-}[u]\geq 0$ when $a\leq 0$ and \(u\geq 0\) in \(\Omega\). Now consider $r>0$ small such that $B_r(x_1)\subset \Omega$ and define
\begin{equation*}
    \varphi = \begin{cases}
    u(x_1) & \text{ in } B_r(x_1) \\ 
    u & \text{ in } \mathbb{R}^{n}\backslash B_r(x_1).
    \end{cases}
\end{equation*}
Clearly, \((u-\varphi)(x) \leq 0\), since \(u(x)\) reaches its maximum at \(x_1\in \Omega\), so that it satisfies \(\mathcal{M}^{-}\varphi(x_1)\geq 0\). Therefore
\begin{align*}
    0\leq\mathcal{M}^{-}[\varphi](x_1) &\leq 2\Lambda C_{n,s}\int_{\mathbb{R}^n}\frac{\varphi(x) - \varphi(x_1)}{|y-x_1|^{n+2s}}\mathrm{d}y \\
    &= 2\Lambda C_{n,s}\int_{\mathbb{R}^n \backslash\Omega}\frac{\varphi(x) - \varphi(x_1)}{|y-x_1|^{n+2s}}\mathrm{d}y + 2\Lambda C_{n,s}\int_{\Omega}\frac{\varphi(x) - \varphi(x_1)}{|y-x_1|^{n+2s}}\mathrm{d}y \\
    & \leq2\Lambda C_{n,s}\int_{\Omega}\frac{\varphi(x) - \varphi(x_1)}{|y-x_1|^{n+2s}}\mathrm{d}y \\
    &= 2\Lambda C_{n,s}\int_{\Omega\setminus B_r(x_1)}\frac{u(y) - u(x_1)}{|y-x_1|^{n+2s}}\mathrm{d}y  <0,
\end{align*}
where the third inequality is due to $u\leq 0$ on $\R^n\setminus \Omega$ and the last one is because $u$ is not constant ($u\leq 0$ on $\partial \Omega$), which in turn leads to a contradiction. This proves the result.
\end{proof}

\begin{lemma}\label{limitação abaixo a(x_0) sublinear}
If $0<q<1$ and $u\le 0 $ in $\mathbb{R}^n\setminus \Omega$, then
    $$
    a(x_0)\geq C_{n,s}\Lambda u(x_0)^{1-q},
    $$
 where $u(x_0)=\max_{\bar\Omega} u$, for some positive constant $C_{n,s}$.  
\end{lemma}

\begin{proof}
Set $u(x_0)=\max_{\bar\Omega} u$, then 
    $$
    a(x_0)u(x_0)^{q}=-\mathcal{M}^{-}[u](x_0)=  -\inf_{\mu} C_{n,s}\int_{\mathbb{R}^n}( u(x+y)+u(x-y)-2u(x_0) )\mu\left(\frac{y}{|y|}\right)\frac{\mathrm{d}y}{|y|^{n+2s}}.
    $$
    Let $\Omega\subset\subset B_R(x_0)$ for some large $R$, then
    $$
    a(x_0)u(x_0)^{q}\geq 2C_{n,s}\Lambda\int_{\mathbb{R}^n\setminus \Omega}\frac{u(x_0)-u(y)}{|y-x_0|^{n+2s}}dy\geq  2C_{n,s}\Lambda  \left(u(x_0)\int_{\mathbb{R}^n\setminus B_R(x_0)} \frac{\mathrm{d}y}{|y-x_0|^{n+2s}}
    - \int_{\mathbb{R}^n\setminus \Omega}\frac{u^+(y)}{|y-x_0|^{n+2s}}\mathrm{d}y\right) .
    $$
Thus, since $u\le 0 $ in $\mathbb{R}^n\setminus \Omega$ and by setting $C_n=C_n(x_0)=\int_{\mathbb{R}^n\setminus B_R(x_0)} \frac{\mathrm{d}y}{|y-x_0|^{n+2s}}$ we have
   $$
    u(x_0)^q(a(x_0)-C_{n,s}\Lambda u(x_0)^{1-q})\geq 0
    $$ 
and hence 
$
    a(x_0)\geq C_{n,s}\Lambda u(x_0)^{1-q}
    $
    as desired.
\end{proof}

\begin{rem}\label{Lemma 1 superlinear}
In the case of a superlinear problem, that is if we had $q>1$, then one would get instead
   \begin{center} 
   $
    \sup_\Omega u\geq \left(\frac{C}{\|a\|_{\infty}}\right)^{\frac{1}{q-1}},
    $\quad 
    where $C(n,s)>0$.
    \end{center}
    Moreover,   
    \begin{center}
    $
    a(x_0)\geq \frac{C_n}{(u(x_0))^{q-1}},\quad
    $
where $u(x_0)=\max_{\bar\Omega} u$.
\end{center}
Indeed, in the proof of the previous lemma one rewrites $$
     u(x_0)(a(x_0)u(x_0)^{q-1}-C_n )\geq 0,
    $$
from which $a(x_0)u(x_0)^{q-1}-C_n\geq 0$ and the two affirmations are verified.
\end{rem}

\begin{rem}
The hypothesis $u\le 0 $ in $\mathbb{R}^n\setminus \Omega$ can be replaced by $$
\int_{\mathbb{R}^n\setminus \Omega}\frac{u(y)}{|y-x_0|^{n+2s}} \mathrm{d}y\leq 0 \textrm{ for all } x_0\in \Omega.
$$
\end{rem}

Our goal is to show that SMP holds when $u^-\neq 0$ has small $L^1_s$ norm, and that it fails when \(\|u^-\|_{L^1_s}\) is large. Since the arguments used to establish the main results of this section rely on compactness, it is first necessary to prove the following lemma, which ensures that the maxima of nontrivial solutions to \eqref{eq:problem} do not degenerate to zero in the limit.

\begin{lemma}\label{Lemma 1 sublinear} Let $\Omega$ be a bounded convex domain.
    Let \(u\) be a nontrivial solution of \eqref{eq:problem} with $0<q<1$ , $u\le 0 $ in $\mathbb{R}^n\setminus \Omega$ and
    $$
    |u(x)|\leq C\sup_\Omega u(1+|x|)^{1+\alpha} \mbox{ in }\mathbb{R}^n.
    $$
    Then
   $
    \sup_\Omega u\geq c>0,
    $
where $c(n,s)>0$.
    
\end{lemma}
\begin{proof}
Assume by contradiction that 
$M_k=\|u_k\|_{L^\infty(\Omega)}\rightarrow 0$. 
By continuity, there exists $\left(x_k\right)_{k \in \mathbb{N}} \in \Omega$ such that $u\left(x_k\right)=M_k$. We define $\mu_k=M_k^{-(q-1) / 2 s} \rightarrow 0, \Omega_k=\left\{y \in \mathbb{R}^N: x_k+\mu_k y \in \Omega\right\}$ and

$$
v_k(y)=M_k^{-1} u_k\left(x_k+\mu_k y\right).
$$

By the boundedness of $\Omega$, we can pass to a subsequence to obtain $x_k \rightarrow x_0 \in \bar{\Omega}$. Recalling that  $\Omega$ is convex, by denoting $d(x_k)=\operatorname{dist}(x_k, \partial \Omega)$, we have $d(x_k) / \mu_k \rightarrow+\infty$, which means $x_0 \in \Omega$ and  $\Omega^k \rightarrow \mathbb{R}^n$.
Here, $\Omega^k \rightarrow \mathbb{R}^n$ in the sense that for every $k, B_{d\left(x_k\right) / \mu_k} \subset \Omega^k$, and $B_{d\left(x_k\right) / \mu_k} \nearrow \mathbb{R}^n$.
Then, $v_k$ is a viscosity solution to

$$
\begin{cases}\mathcal{M}^{-}[v_k] + a_{k}v_{k}^q=0 & \text { in } \Omega^k \\ v_k\geq 0 & \text { in } \Omega^k\end{cases}
$$
with $a_k(y)=a(x_k+\mu_k y)$ and $v_k \leq 1, v_k(0)=1$.  Moreover, by Lemma \ref{limitação abaixo a(x_0) sublinear}, 
$$
a(x_k)=a_k(0)\geq C_nv_k(0)^{1-q}=C_n.
$$  

We next prove that, up to a subsequence, $v_k \rightarrow v$ locally uniformly in $\mathbb{R}^n$.
Indeed, fix a compact $K \subset \mathbb{R}^n$ and a smooth domain $U$ such that $K \subset \subset U$. For some $k_0, K \subset \subset U \subset \subset \Omega^k$ for any $k \geq k_0$, so we have
$$
\mathcal{M}^{-}[v_k] + a_{k}v_{k}^q=0 \text { in } U
$$
in the viscosity sense. 
Besides,
$
\|v_k\|_{L^1_s} \le C\,\, \|u_k\|_{L^1_s}.$
We have $v_k \in C^\beta(K)$ and $\left\|v_k\right\|_{C^\beta(K)} \leq C_0(\left\|v_k\right\|_{L^{\infty}(U)}+\|v_k\|_{L^1_s}) \leq 2C_0$ by interior regularity estimates. 
 Then there exists a subsequence $\left(v_k\right)$ such that $v_k \rightarrow v$ in $C(K)$ by the compact embedding of $C^\beta(K)$ into $C(K)$. Therefore $v_k \rightarrow v$ pointwise and as  
$$
|v_k(x)|\leq C(1+|x|)^{1+\alpha},
$$
 the dominated convergence theorem gives us $v_k \rightarrow v$ in $L_{s}^1\left(\mathbb{R}^n\right)$.

Finally, by stability of viscosity solutions under $C_{\text {loc }}+L_{s}^1\left(\mathbb{R}^n\right)$ limits, $v$ is a viscosity solution of
$$
\mathcal{M}^{-}[v] + a(x_0)v^p=0 \text { in } \mathbb{R}^n.
$$
By Lemma \ref{lemma 5.1} the maximum of $u_k$ can only be reached at a point where $a>0$ and $a(x_k)\geq C_n$, then $a(x_0)> 0$ and we obtain that $\mathcal{M}^{-}[v]\leq 0$, also $v(0)=1$. Since $v_k\geq 0$ in $\Omega_k$ we have that $v\geq 0$ in $\mathbb{R}^n$, therefore, by SMP, $v>0$. Then a Liouville theorem applies, which follows in the sublinear case along the same lines as in \cite {FQ2011}, as carried out in \cite{NPQ2024} for conical domains. That is, there exists no positive supersolutions to the equation if $0<q<1$. In particular, there exists no positive solution, and we are done.
\end{proof}

\begin{rem}\label{blow-up} In the case of a superlinear problem,  \(\| u_{k}\|_{L^\infty(\Omega)} \leq C\), for all \(k\), whenever $\Omega \subset \mathbb{R}^N$ is a bounded convex domain of class $C^2$, and $q \in\left(1, q^*\right)$, for some $q^*$.

Indeed, if by contradiction $M_k=\|u_k\|_{L^\infty(\Omega)}\rightarrow \infty$, by continuity there exists $\left(x_k\right)_{k \in \mathbb{N}} \in \Omega$ such that $u\left(x_k\right)=M_k$. We define $\mu_k=M_k^{-(q-1) / 2 s} \rightarrow 0, \Omega_k=\left\{y \in \mathbb{R}^n: x_k+\mu_k y \in \Omega\right\}$ and
$
v_k(y)=M_k^{-1} u_k\left(x_k+\mu_k y\right).
$
Then $v_k$ is a viscosity solution to
$$
\begin{cases}\mathcal{M}^{-} [v_k] + a_{k}v_{k}^q=0 & \text { in } \Omega^k \\ v_k\geq 0 & \text { in } \Omega^k\end{cases}
$$
with $a_k(y)=a(x_k+\mu_k y)$ and $v_k \leq 1$ in $\Omega_k$, $ v_k(0)=1$. Also, by Lemma~\ref{Lemma 1 superlinear},
$$
a(x_k)=a_k(0)\geq \frac{C_n}{(v_k(0))^{q-1}}=C_n.
$$ 
As before, one proves that, up to a subsequence, $v_k \rightarrow v$ locally uniformly in $\mathbb{R}^n$, where
 $v$ is a viscosity solution of
$
\mathcal{M}^{-}[v] + a(x_0)v^p=0 \text { in } \mathbb{R}^n.
$
By Remark \ref{Lemma 1 superlinear}, the maximum of $u_k$ can only be reached at a point where $a>0$ and $a(x_k)>C_n>0$ then $a(x_0)> 0$ and we obtain that $\mathcal{M}^{-} [v]\leq 0$, $v(0)=1$. Since $v_k\geq 0$ in $\Omega_k$ we have that $v\geq 0$ in $\mathbb{R}^n$, therefore, by SMP, $v>0$. By the Liouville result in \cite{FQ2011}, there exists a critical exponent $q^*>1$ for which there exists no positive supersolutions to the equation if $1<q<q^*$.
\end{rem}

Now we are in a position of proving the main result of the paper.

\begin{proof}[Proof of Theorem \ref{teo1}]
 Assume, by contradiction, that \eqref{eq:problem} admits a sequence of viscosity solutions \(u_k\) such that \(u_k\notin\mathcal{P}^{0}\) for every \(k\), with $\|u_k^{-}\|_{L^{1}_{s}}\rightarrow 0$.

First, suppose that \(\| u_{k}\|_{\infty} \leq C\), for all \(k\). By \(C^{1,\alpha}\) regularity estimates (see \cite{birindelli2014}, \cite{imbert2013c1}), we obtain \(u_{k} \to u_{0}\) locally uniformly in \(C^{1,\alpha}(\Omega)\). By Lemma~\ref{Lemma 1 sublinear}, we know that there exists $c>0$ such that, for every \(k\in\mathbb{N}\), \[\sup_\Omega u_k \geq c >0 .\]
Therefore, we conclude that $u_0\geq 0$ and it is nontrivial. By stability (see \cite{caffarelli2011regularity}) we have 

\begin{equation*}
    \begin{cases}
    \mathcal{M}^{-}[u_0] \leq 0 & \text{ in } \Omega, \\ 
    u_0\geq0 & \text{ on }\Omega,
    \end{cases}
\end{equation*}
so $u_0>0$ in $\Omega$. As $\|u^-\|_{L^1_s}= 0$ and $\dfrac{\alpha _{r}}{r^{2s}}>0$ for every $x_0\in \partial\Omega$, we can apply Theorem~\ref{thm1} to ensure that
\begin{equation}
        \liminf_{x\to x_0, x\in\Omega}\frac{u(x)}{|x-x_{0}|^{s}}>0.
\end{equation}
On the other hand, we know that (see \cite[Proposition~1.1]{ros2016boundary})

\begin{center}
    $\frac{|u_k|}{d^{s}}\to \frac{|u_0|}{d^{s}}\;\text{ in $\Omega$ for \(k\) large enough,}$
\end{center}
where $d=\mathrm{dist}(x,\partial \Omega)$. Therefore, 
\begin{equation}
        \liminf_{x\to x_0, x\in\Omega}\frac{u_k(x)}{|x-x_{0}|^{s}}>0.
\end{equation}
Furthermore, since $u_k>0$ for sufficiently large $k$, it follows that $u_k\in\mathcal{P}^{0}(\Omega)$, which contradicts our assumption. Hence, the result follows.

If, otherwise, $\| u_{k}\|_{\infty} \rightarrow\infty$, we consider $v_k=\frac{u_k}{\| u_{k}\|_{\infty}}$, so $\| v_{k}\|_{\infty}= 1$ and
 \begin{equation*}
    \begin{cases}
    \mathcal{M}^{-}[v_k] =\| u_{k}\|_{\infty}^{q-1}av_k^q & \text{ in } \Omega, \\ 
    v_k\geq0 & \text{ on }\Omega,
    \end{cases}
\end{equation*}
so we obtain \(v_{k} \to v_{0}\) locally uniformly in \(C^{1,\alpha}(\Omega)\), therefore
 \begin{equation*}
    \begin{cases}
    \mathcal{M}^{-}[v_0] =0 & \text{ in } \Omega, \\ 
    v_k\geq0 & \text{ on }\Omega,
    \end{cases}
\end{equation*}
and the proof follows as in the first case.
\end{proof}

\begin{rem}
Although SMP in the superlinear regime follows from the sublinear one, a direct proof can be applied exactly as in the sublinear case, by replacing  Lemma~\ref{Lemma 1 sublinear} by Remark~\ref{Lemma 1 superlinear}.
\end{rem}

\begin{proof}[Proof of Theorem \ref{teo2}]
By following the same steps of the proof of the Theorem \ref{teo1}, we conclude the same result without assuming the exterior sign condition on $u$ whenever $a$ has a positive sign in $\Omega$. Indeed, Lemma \ref{lemma 5.1} holds directly and  $a(x)>c>0$ for $x\in \Omega'\Subset \Omega$, therefore we have a similar result of Lemma \ref{limitação abaixo a(x_0) sublinear}. This proves Theorem \ref{teo2}.
\end{proof}


\section{Existence of nontrivial solutions and a solution that violates SMP}

We start the section with presenting and proving an existence result as far as $a^+\not\equiv 0$. Next we will show that SMP fails when \(\|u^-\|_{L^1_s}\) is large.

We first recall the following result concerning the existence of the principal eigenvalue in general domains for  nonlocal elliptic equations, whose proof follows for instance from Theorem 1.1 in \cite{davila2023harnack}.

\begin{proposition}\label{Erwin}
    Let \(\Omega\subset\mathbb{R}^{n}\), \(s\in(0,1)\) and \(\mathcal{I}=\mathcal{M}^{\pm}\). Then there exists an eigenpair \((\phi^{+},\lambda^{+})\) with \(\phi^{+}>0\) in \(\Omega\), solving the problem 
    \begin{equation*}
    \begin{cases}
    \mathcal{I}(\phi)  = -\lambda\phi & \text{ in } \Omega, \\ 
    \phi = 0 & \text{ in }\mathbb{R}^{n}\backslash\Omega,
    \end{cases}
    \end{equation*}
    in the viscosity sense. The principal eigenvalue \(\lambda^{+}\geq 0\) is characterized by
    \begin{equation*}
        \lambda^{+} = \sup\{\lambda:\exists\phi>0\text{ in }\Omega, \phi\geq0\text{ in }\mathbb{R}^{n},\text{ s.t. }\mathcal{I}\phi  \leq -\lambda\phi\text{ in }\Omega\}.
    \end{equation*}
\end{proposition}

\begin{theorem}\label{TheoremExistence}
Let $q\in (0,1)$, and \(a\in C(\overline{\Omega})\) such that $a^+\not\equiv 0$, and let $g\in L^{1}_{s}(\mathbb{R}^n)$ with $g\ge 0$ in a neighborhood of the boundary.  
Then there exists $\delta>0$ so that, whenever $\|g^{-}\|_{L^{1}_{s}(\mathbf{R}^n)}< \delta $, there exists a nontrivial solution of 
\begin{equation}\label{P-zero-fora}
    \begin{cases}
        \mathcal{M}^{\pm}[u]+a(x)u^{q} = 0 & \text{ in } \Omega, \\
        u=g & \text{ in } \mathbb{R}^{n}\setminus\Omega.
    \end{cases}
\end{equation} 
If, in particular, $a\ge a_0 >0$ in $\Omega$, then there always exist a nontrivial solution of \eqref{P-zero-fora}, for any $g$.
\end{theorem}
Observe that when $g\equiv 0$ we have the existence of positive solutions via the SMP, as mentioned in the introduction.

\begin{proof}
We consider the extremal operator $\mathcal{M}^-$, since for $\mathcal{M}^+$ it is analogous. 
We take a ball $B_0$ such that \(\overline{{B}}_0\subset\{x\in \Omega : a(x)>0\}\), with \(a\geq a_{0}>0\) in \({B_0}\) and let \((\lambda_{1}, \phi_{1})\) be the first eigenpair of 
\begin{equation*}
    \begin{cases}
        \mathcal{M}^{-}[\phi_1] + \lambda_{1}\phi_{1} = 0 & \text{ in } \Omega, \\
        \phi_{1} > 0 & \text{ in } \Omega, \\
        \phi_{1} = 0 & \text{ in } \mathbb{R}^{n}\setminus\Omega,
    \end{cases}
\end{equation*} 
normalized so that $\|\phi_{1}\|_{\infty}=1$. The idea is to construct a pair of sub and supersolutions for \eqref{P-zero-fora}.

\begin{claim}\label{Claim1-ex}
Given a fixed $\epsilon\in (0,  (\frac{ a_{0}}{\lambda_{1}})^{\frac{1}{1-q}}]$, the function 
    \begin{equation*}
    \underline{u} =
    \begin{cases}
        \epsilon \phi_{1} & \text{ in } B_0 \\[4pt]
        0 & \text{ in } \Omega \setminus B_0\\
         g & \text{ in } \mathbb{R}^{n} \setminus \Omega
    \end{cases}
\end{equation*}
    is a viscosity subsolution of \eqref{P-zero-fora} in \(\Omega\).
\end{claim}
 
To prove the claim, we first observe that the function \(\epsilon\phi_1\) serves as subsolution for the equation in $B_0$. Indeed, \[-\mathcal{M}^{-}[\epsilon\phi_1] = -\epsilon \mathcal{M}^{-}[\phi_1] = \lambda_{1}\epsilon \phi_{1} \leq a(x)(\epsilon \phi_{1})^{q}\quad\text{ in }B_0,\]
 since 
$\lambda_{1}\epsilon^{1-q}\phi_{1}^{1-q}\leq\lambda_{1}\epsilon^{1-q}\leq a_{0}\leq a(x).$

Observe that in the special case $a\ge a_0 >0$ in $\Omega$, we take $B_0=\Omega$ by already obtaining that $\underline{u}$ is a viscosity subsolution in $\Omega$.

In the general case, let \(\varphi \in C^{2}(\Omega)\) such that \(\underline{u}\leq \varphi\) in \(\mathbb{R}^n\) and \(\underline{u}(x_0) = \varphi(x_0)\). If \(x_0\in \Omega\setminus B_0\), then \(\underline{u}(x_0)=\varphi(x_0)=0\), and

\[
\int_{\R^n} \frac{\varphi (y) \rmd y }{|y-x_0|^{n+2s}} 
\ge 
\int_{\R^n\setminus\Omega} \frac{g (y) \rmd y }{|y-x_0|^{n+2s}} 
+ \epsilon \int_{B_0} \frac{\phi_1 (y)\rmd y }{|y-x_0|^{n+2s}} .\] 
By writing $g=g^+-g^-$ and $\phi_1\ge C_0$ in $B_0$, we may ensure that
$$
\int_{\R^n\setminus\Omega} \frac{g^- (y) \rmd y }{|y-x_0|^{n+2s}} 
 \le c_0 \|g^{-}\|_{L^{1}_{s}(\mathbb{R}^n)}
\le \epsilon \int_{B_0} \frac{\phi_1 (y)\rmd y }{|y-x_0|^{n+2s}} 
$$
as far as the integral on the LHS is finite by our hypothesis of positivity of $g$ near $\partial\Omega$, while the integral on the RHS, say $I$, is treated in the following way: 
either $x_0\in \partial B_0 $ and so $I=\infty$ and we are done; 
or $x_0\not\in \partial B_0 $ with $\mathrm{dist}(x_0, \partial B_0)=d_0$. 
For the latter, since $I\to \infty $ as $x\to x_0$, there exists $\delta_0$ such that $I\ge 1$ for $d_0 \le \delta_0$. 
If  $d_0\ge \delta_0$ then $I\ge I_0 (\delta_0,\phi_1,B_0)>0$ and we choose $\delta$ small such that $\delta \le \frac{\epsilon \max (1, I_0 ) }{c_0}$.  

\smallskip

Therefore we conclude $\mathcal{M}^{-}[\varphi](x_0) 
\ge  0 =-a(x_0)\underline{u}^{q}(x_0), $
for $x_0 \in \Omega\setminus B_0$.
\medskip

Now, if \(x_0\in B_0\), then \(\underline{u}(x_0) = \epsilon\phi_{1}(x_0)>0\), it follows that \(\varphi(x_0) = \underline{u}(x_0)>0\) and \(\epsilon\phi_1\leq \varphi\) in \(B_{r}(x_0)\subset B_0\) for \(r>0\) small enough. Since \(\epsilon\phi_1\) is subsolution in \(B_0\), and \(\varphi\) touches \(\epsilon\phi_1\) from above, then by definition of viscosity subsolution,
\[\mathcal{M}^{-}[\varphi](x_0) + a(x_0)\underline{u}^{q}(x_0) = \mathcal{M}^{-}[\varphi](x_0) + a(x_0)(\epsilon \phi_1)^{q}(x_0)\geq 0.\]
This proves the Claim \ref{Claim1-ex}.

In order to construct a viscosity supersolution of~\eqref{P-zero-fora}, we consider \(\Psi\) solution of
\begin{equation*}
    \begin{cases}
        \mathcal{M}^{-}[\Psi] = -\|a\|_{\infty} & \text{ in } \Omega, \\[4pt]
        \Psi = g & \text{ in } \mathbb{R}^{n} \setminus \Omega.
    \end{cases}
\end{equation*}
Observe that $\Psi\ge 0$ in $\Omega$ by the maximum principle. Then we get further $\Psi\ge 0$ by our strong maximum principle, by our hypothesis on $g$. Moreover, $\Psi $ is bounded, for instance by the ABP estimate given by \cite[Theorem 3.1]{Kitano}.

\begin{claim}\label{Claim2-ex}
Given a fixed \(k \) so that \(\Psi\leq k^{\frac{1-q}{q}}\), the function \(\overline{u}=k\Psi\) is a supersolution of~\eqref{P-zero-fora}.
\end{claim}

Indeed, by \(k^{1-q}>\Psi^{q}\) and using \(a(x)\leq \|a\|_{\infty}\), we have \[-\mathcal{M}^{-}[k\Psi] = k\|a\|_{\infty} \geq (k\Psi)^{q}\|a\|_{\infty} \geq (k\Psi)^{q}a(x),\] as desired. This proves the Claim \ref{Claim2-ex}.

Finally, we conclude that there exists a viscosity solution \(u\) of~\eqref{P-zero-fora} such that \(\underline{u}\leq u \leq \overline{u},\) in particular \(u\geq0\) is nontrivial. 
This completes the proof of the theorem.
\end{proof}

To finish the section, we provide an example of solutions to problem \eqref{eq:problem} that do not satisfy the strong maximum principle, as long as the assumption of Theorem \ref{teo2} on the smallness of the norm of the negative part of \(u\) is not fulfilled. 
The idea is to start with a positive solution and gradually distort it until it touches zero for the first time, always ensuring that there remains a region where it stays positive. We follow the construction from \cite{Kassmann}; see also \cite{bucur2016nonlocal}.

\begin{example}[When \(\|u^-\|_{L^1_s}\rightarrow \infty\)]\label{counterexample}
Consider a nontrivial solution of the problem
\begin{equation*}
\begin{cases}
    \Delta^su_M(x) + a(x)u_M^{q}(x) = 0 & \text{ in } B_1, \\ 
    u_M=0 & \text{ on }B_2\setminus B_1,\\
    u_M= 1-M & \text{ in } B_3\setminus B_2, \\
    u_M=1  & \text{ in } \R^n\setminus B_3,
\end{cases}
\end{equation*}
given by Theorem \ref{TheoremExistence} for a bounded potential such that $a(x)\geq a_0>0$ in $\Omega$.

Observe that by choosing $M$ sufficiently negative, we ensure that \(1-M\) is positive, which implies that $u_{M}>0$ in $B_1$. 

We will show that as $M\rightarrow \infty$, we have $\inf_{B_{1/2}} u_M\rightarrow -\infty$. Thus, the smallest $M_*$ such that $u_{M_*}(x)=0$ for some point $x\in B_1$ will be our example.

Suppose, by contradiction, that $\inf_{B_{1/2}} u_M > -\ell$ for all $M$. We define
\[
v_M(x)=\frac{u_M+M-1}{M},
\]
and observe that
\[
\Delta^s v_M(x) + \bar{a}(x)v_M^{q}(x) \geq c_0\bar{a}(x)\left(\frac{M-1}{M}\right)^q \geq 0 \quad \text{in } B_1,
\]
where $\bar{a}(x)=\frac{M^{q-1}a(x)}{c_0}$. Hence,
\begin{equation*}
\begin{cases}
    \Delta^s v_\infty \geq  0 & \text{ in } B_1, \\ 
    v_\infty = 1 & \text{ on } B_2\setminus B_1,\\
    v_\infty = 0 & \text{ in } B_3\setminus B_2, \\
    v_\infty = 1 & \text{ in } \R^n\setminus B_3
\end{cases}
\end{equation*}

Thus, by the maximum principle, $v_\infty \leq 1$ in $B_1$. On the other hand, by the definition of $v_M$ and the contradiction hypothesis, $v_\infty \geq 1$ in $B_{1/2}$. Therefore, $v_\infty$ attains its maximum at some $x_*\in B_{1/2}$ with $v_\infty(x_*)\geq 1$. Applying the equation yields
\[
0 \leq \int_{\R^n}\frac{v_\infty(y)-v_\infty(x_*)}{|y-x_*|^{n+2s}}\mathrm{d}y \leq \int_{B_3\setminus B_2}\frac{v_\infty(y)-v_\infty(x_*)}{|y-x_*|^{n+2s}}\mathrm{d}y < 0,
\]
This implies that \(\inf_{B_{1/2}}u_{M}\to -\infty\) as \(M\to\infty\). In particular, there exists a value \(\widetilde{M}>0\) such that \(\inf_{B_{1/2}}u_{\widetilde{M}}=0\), then, as we increase \(M\), \(\widetilde{M}\) is the first value such that \(u_{M}\) is no longer positive and reach a zero value. This \(u_{\widetilde{M}}\) is obtained as desired.
\end{example}


\end{document}